\documentclass[a4paper,12pt]{article}

\usepackage{mathrsfs}
\usepackage{amssymb}
\usepackage{amsmath}
\usepackage{amsthm}
\newtheorem{pro}{Proposition}[section]
\newtheorem{df}{Definition}[section]
\newtheorem{thm}{Theorem}[section]

\newtheorem{la}{Lemma}[section]
\newtheorem{re}{Remark}[section]

\begin{document}
\title{\large\textbf{Reflected Backward Stochastic Differential Equations Driven by
L\'{e}vy Process  }}

\author{\textbf { Yong Ren$^1$,\footnote{
Correspondence author,  E-mail:brightry@hotmail.com} \hspace{0.5cm}
Xiliang Fan$^2$}
\\
 1. School of Mathematics and Physics, University of Tasmania,\\
GPO
Box 252C-37, Hobart, Tasmania 7001, Australia \\
2. Department of Mathematics, \\
Anhui Normal University,Wuhu
241000,China \\
}
\date{}
\maketitle

\textbf{Abstract:}\ In this paper, we deal with a class of reflected
backward stochastic differential equations associated to the
subdifferential operator of a lower semi-continuous convex function
driven by Teugels martingales associated with L\'{e}vy process. We
obtain the existence and uniqueness of solutions to these equations
by means of the penalization method. As its application, we give a
probabilistic interpretation for the solutions of a class of partial differential-integral inclusions. 

\section{Introduction}
On the one hand, partial differential-integral inclusions (PDIIs in
short) play an important role in characterizing many social,
physical, biological and engineering problems, one can see
Balasubramaniam and Ntouyas \cite{B} and references therein. On the
other hand, reflected backward stochastic differential equations
(RBSDEs in short) associated to a multivalued maximal monotone
operator defined by the subdifferential of a convex function has
first been introduced by Gegout-Petit \cite{Geg}. Further, Pardoux
and R\^{a}canu \cite{PR} gave the existence and uniqueness of the
solution of RBSDEs, on a random (possibly infinite) time interval,
involving a subdifferential operator was proved in order to to give
the probabilistic interpretation for the viscosity solution of some
parabolic and elliptic variational inequalities. Following, Ouknine
\cite{Ou}, N$^,$Zi and Ouknine \cite{Mn}, Bahlali et al.
\cite{Ba1,Ba2} discussed this type RBSDEs driven by Brownian motion
or the combination of Brownian motion and Poisson random measure
under the conditions of Lipschitz, locally Lipschitz or some
monotone conditions on the coefficients.  El Karoui et al. \cite{El}
have got another type RBSDEs different here, where one of the
components of the solution is forced to stay above a given barrier,
which provided a probabilistic formula for the viscosity solution of
an obstacle problem for a parabolic PDE. Since then, there were many
works on this topic. One can see Matoussi \cite{Ma}, Hamed\`{e}ne
\cite{Ha1,Ha2}, Lepeltier and Xu \cite{L}, Ren et al.
\cite{Ren1,Ren3} and so on.

The main tool in the theory of BSDEs is the martingale
representation theorem, which is well known for martingale which
adapted to the filtration of the Brownian motion or that of Poisson
point process (Pardoux and Peng \cite{PP1}, Tang and Li \cite{Ta}).
Recently, Nualart and Schoutens \cite{Nu1} gave a martingale
representation theorem associated to L\'{e}vy process. Furthermore,
they showed the existence and uniqueness of solutions to BSDEs
driven by Teugels martingales associated with L\'{e}vy process with
moments of all orders in \cite{Nu2}. The results were important from
a pure mathematical point of view as well as in the world of
finance. It could be used for the purpose of option pricing in a
L\'{e}vy market and related partial differential equation which
provided an analogue of the famous Black-Scholes partial
differential equation.

Motivated by the above works, the purpose of the present paper is to
consider RBSDEs associated to the subdifferential operator of a
lower semi-continuous convex function driven by Teugels martingales
associated with L\'{e}vy process which considered in Nualart and
Schoutens \cite{Nu1,Nu2}. We obtain the existence and uniqueness of
the solutions for RBSDEs. As its application, we give a
probabilistic interpretation for a class of PDIIs.

The paper is organized as follows. In Section 2, we introduce some
preliminaries and notations. Section 3 is devoted to the proof of
the existence and uniqueness of the solutions to RBSDEs driven by
L\'{e}vy processes by means of the penalization methods. A
probabilistic interpretation for a class of PDIIs by our RBSDEs is
given in the last section.
\section{Preliminaries and notations}

Let $T>0$ be a fixed terminal time and $(\Omega, \mathcal{F},P)$ be
a complete probability space. Let
  $\{L_t:t\in [0,T]\}$ be a $\mathbb{R}$-valued L\'{e}vy process corresponding to a standard L\'{e}vy measure
  $\nu$ whose characteristic function has the
following form:
\begin{center}
$E(e^{iuL_t})=exp[iaut-\frac{1}{2}\sigma^2u^2t+t\int_{\mathbb{R}}(e^{iux}-1-iux1_{\{|x|<1\}})\nu(dx)],$
\end{center}
where $a\in \mathbb{R}, \sigma\geq 0$. Furthermore, L\'{e}vy measure
$\nu$ satisfying the following conditions:
  \par
 (1) $\int_{\mathbb{R}}(1\wedge y^2)\nu(dy)< \infty;$
 \par
 (2) $\int_{]-\varepsilon,\varepsilon[^c}e^{\lambda |y|}\nu(dy)< \infty,$ for every $\varepsilon>0$
  and for some $\lambda>0$.

This shows that $L_t$ has moments of all orders. We denote by
$(H^{(i)})_{i\geq 1}$ the linear combination of so called Teugels
Martingale $Y_t^{(i)}$ defined as follows associated with the
L\'{e}vy process
  $\{L_t:t\in [0,T]\}$. More precisely
  $$H^{(i)}_t=c_{i,i}Y_t^{(i)}+c_{i,i-1}Y_t^{(i-1)}+\cdots+c_{i,1}Y_t^{(1)},$$
  where $Y_t^{(i)}=L_t^{(i)}-E[L_t^{(i)}]=L_t^{(i)}-tE[L_1^{(i)}]$ for all $i\geq 1$ and $L_t^{(i)}$ are
  power-jump processes.
  That is, $L_t^{(i)}=L_t$ and $L_t^{(i)}=\sum_{0<s\leq t}(\triangle L_t)^i$ for $i \geq 2$. It was shown
  in Naulart and
  Schoutens \cite{Nu1} that the coefficients $c_{i,k}$ correspond to the orthonormalization of the polynomials
  $1,x,x^2,\cdots$
  with respect to the measure $\mu(dx)=x^2\nu(dx)+\sigma^2\delta_0(dx)$:
  $$q_{i-1}=c_{i,i}x^{i-1}+c_{i,i-1}x^{i-2}+\cdots+c_{i,1}.$$
  The martingale $(H^{(i)})_{i\geq 1}$ can be chosen to be pairwise strongly orthonormal
  martingale. Furthermore, $[H^{(i)},H^{(j)}], i\neq j,$ and
  $\{[H^{(i)},H^{(j)}]_t-t\}_{t\geq 0}$ are uniformly integrable
  martingales with initial value 0, i.e.
  $<H^{(i)},H{(j)}>_t=\delta_{ij}t.$

\begin{re} If $\nu=0$, we are in the classic Brownian case and all
non-zero degree polynomials $q_i(x)$ will vanish, giving
$H_t^{(i)}=0,i=2,3,\cdots.$ If $\mu$ only has mass at 1, we are in
the Poisson case; here also $H_t^{(i)}=0,i=2,3,\cdots.$ Both cases
are degenerate in this L\'{e}vy framework.
\end{re}

Let $\mathcal{N}$ denote the totality of $P-$null sets of
$\mathcal{F}$. For each $t\in [0,T]$, we define
  \begin{center}
  $\mathcal{F}_t \triangleq \sigma(L_s,0\leq s\leq t)\vee \mathcal{N}.$
  \end{center}
Let us introduce some spaces:
\par
$\bullet \mathcal{H}^2=\{\varphi:\mathcal{F}_t-$progressively
measurable process, real-valued process, s.t.
$E\int_0^T|\varphi_t|^2dt<\infty\}$ and denote by $\mathcal{P}$$^2$
the subspace of $\mathcal{H}$$^2$ formed by the predictable
processes;
\par
$\bullet S^2=\{\varphi:\mathcal{F}_t-$progressively measurable
process, real-valued process, s.t. $E(\sup_{0\leq t \leq
T}|\varphi(t)|^2)< \infty \};$
\par
 $\bullet l^2=\{(x_n)_{n\geq 0}: $be real valued sequences such that $\sum_{i=0}^\infty x_i^2\leq \infty\}$. We shall denote by $\mathcal{H}$$^2(l^2)$ and $\mathcal{P}$$^2(l^2)$ the corresponding spaces of
  $l^2$-valued process equipped with the norm
  $$||\varphi||^2=\sum_{i=0}^\infty E\int_0^T|\varphi_t^{(i)}|^2dt.$$

Now, we give the following assumptions:
 \par
 (H1) The terminal value $\xi\in L^2(\Omega,\mathcal{F}$$_T,P)$;
 \par
 (H2) The coefficient $f:[0,T]\times \Omega \times R$$\times l^2
 \rightarrow R$ is $\mathcal{F}_t-$progressively measurable, such that
 $f(\cdot,0,0)\in \mathcal{H}^2;$

\par
 (H3) There exists a constant $C>0$  such
 that for every $(\omega,t)\in \Omega\times
 [0,T],(y_1,z_1),(y_2,z_2)\in R\times l^2$
\par
$|f(t,y_1,z_1)-f(t,y_2,z_2)|^2\leq C(|y_1-y_2|^2+||z_1-z_2||^2);$

(H4) Let $\Phi:R\rightarrow (-\infty,+\infty]$ be a proper lower
semicontinuous convex function.

Define:

$Dom(\Phi) := \{x\in R :\Phi(x)<+\infty \},$

$\partial \Phi(x) := \{x^\ast\in R:<x^\ast,x-u>+\Phi(u)\leq
\Phi(x),\forall x\in R\},$

$Dom(\partial \Phi)) := \{x\in R:\partial \Phi \neq \emptyset\},$

$Gr(\partial \Phi)) := \{(x,x^\ast)\in R^2: x\in Dom (\partial
\Phi)), x^\ast \in \partial \Phi (x)\},$

Now, we introduce a multivalued maximal monotone operator on $R$
defined by the subdifferential of the above function $\Phi$. The
details appeared in Brezis \cite{Br}.

For all $x\in R$, define

$\Phi_n(x)=\min_{y}(\frac{n}{2}|x-y|^2+\Phi(y)).$

Let $J_n(x)$ be the unique solution of the differential inclusion
$x\in J_n(x)+\frac{1}{n}\partial \Phi(J_n(x)),$  which is called the
resolvent of the monotone operator $A=\partial \Phi$. Then, we have
\begin{pro}
(1) $\Phi_n: R\rightarrow R$ is Lipschitz continuous;

(2) The Yosida approximation of $\partial \Phi$ is defined by
$A_n(x):=\nabla \Phi_n(x)=n(x-J_n(x)), x\in R$ which is monotonic
and Lipschitz continuous and there exists $a\in
\mbox{interior}(Dom(\Phi))$ and positive numbers $R,C$ satisfies
\begin{equation}
<\nabla \Phi_n(z)^\ast,z-a>\geq R|A_n(z)|-C|z|-C,\ \mbox{for all} \
z\in R\ \mbox{and}\ n \in N;
\end{equation}
(3) For all $x\in R$, $A_n(x)\in A(J_n(x)).$
\end{pro}

Further, we assure that $\xi \in \overline{Dom(\Phi)}$ and
$E\Phi(\xi)<\infty.$

 This paper is mainly discuss the following
reflected backward stochastic differential equations. In doing so,
we first give its definition.
 \begin{df}
 By definition a solution associated with the above assumptions $(\xi,f,\Phi)$ is
 a triple
 $(Y_t,Z_t,K_t)_{0\leq t \leq T}$ of progressively measurable processes  such that
\begin{equation}
\begin{array}{rcl}Y_t
&=&\xi+\int_t^Tf(s,Y_s,Z_s)ds+K_T-K_t-\sum_{i=1}^\infty\int_t^TZ_s^{(i)}dH_s^{(i)},\
0\leq t \leq T ,\end{array}
\end{equation}
and satisfying

(1) $ (Y_t,Z_t)_{0\leq t \leq T}\in
 S^2\times \mathcal{P}^2(l^2)\ \mbox{and}\ \{Y_t, 0\leq t \leq T\}$ is c\`{a}dl\`{a}g (right continuous with left limits) and take
values in $\overline{Dom (\Phi)}$;

(2) $\{K_t, 0\leq t \leq T\}$ is absolutely continuous, $K_0=0$, and
for all progressively measurable and right continuous process
$\{(\alpha_t,\beta_t),0\leq t \leq T\}\in Gr(\partial \Phi),$ we
have \begin{center} $\int_0^\cdot(Y_t-\alpha_t)(dK_t+\beta_tdt)\leq
0.$\end{center}
\end{df}
In order to obtain the existence and uniqueness of the solutions to
(2), we consider the following penalization form of (2):
\begin{equation}
\begin{array}{rcl}Y_t^n
&=&\xi+\int_t^T[f(s,Y_s^n,Z_s^n)-A_n(Y_s^n)]ds-\sum_{i=1}^\infty\int_t^TZ_s^{(i),n}dH_s^{(i)},\
0\leq t \leq T ,\end{array}
\end{equation}
where $\xi, f$ satisfies the assumptions stated above and $A_n$ is
the Yosida approximation of the operator $A=\partial \Phi.$ Since
$A_n$ is Lipschitz continuous, it is known from the result of
 \cite{Nu2}, that Eq. (3) has a unique solution.

Set $K_t^n=-\int_0^t A_n(Y_s^n), \ 0\leq t \leq T.$ Our aim is to
prove the sequence $(Y^n,Z^n,K^n)$ convergence to a sequence
$(Y,Z,K)$ which is our desired solution to RBSDEs (2).
\section{Existence and uniqueness of the solutions}
The principal result of the paper is the following theorem.
\begin{thm}
Assume that assumptions on $\xi,f$ and $\Phi$ hold, the RBSDEs (2)
has a unique solution $(Y_t,Z_t,K_t)_{0\leq t \leq T}.$ Moreover,

$\lim_{n\rightarrow \infty}E\sup_{0\leq t\leq T}|Y_t^n-Y_t|^2=0,$

$\lim_{n\rightarrow \infty}E\int_0^T||Z_t^n-Z_t||^2dt=0,$

$\lim_{n\rightarrow \infty}E\sup_{0\leq t\leq T}|K_t^n-K_t|^2=0,$

where $(Y^n,Z^n)$ be the solution of Eq. (3).
\end{thm}
In the sequel, $C>0$ denotes a constant which can change from line
to line.

 The proof of the Theorem is divided into the following Lemmas.
\begin{la}
Under the assumptions of Theorem 3.1, there exists a constant
$C_1>0$ such that for all $n\geq 1$
\begin{center}
$E(\sup_{0\leq t \leq
T}|Y_t^n|^2+\int_0^T|Z_s^n|^2ds+\int_0^T|A_n(Y_s^n)|ds)\leq C_1.$
\end{center}
\end{la}
\begin{proof}
By It\^{o} formula, we have
\begin{equation}\begin{array}{rcl}|Y_t^n-a|^2&=&|\xi-a|^2+2\int_t^T(Y_s^n-a)f(s,Y_s^n,Z_s^n)ds\\
&&-2\int_t^T(Y_s^n-a)A_n(Y_s^n)ds-\int_t^T||Z_s^n||^2ds\\
&&-2\sum_{i=1}^\infty
\int_t^T(Y_s^n-a)Z_s^{(i),n}dH_s^i.
 \end{array}\end{equation}
Taking expectation on the both sides and considering (1), we obtain

$E|Y_t^n-a|^2+E\int_t^T||Z_s^n||^2ds$
\begin{equation}\begin{array}{rcl}&\leq&E |\xi-a|^2+2E\int_t^T(Y_s^n-a)f(s,Y_s^n,Z_s^n)ds\\
&&-2RE\int_t^T|A_n(Y_s^n)|ds+2CE\int_t^T|Y_s^n|ds+2C.
 \end{array}\end{equation}
Further, we get

$E|Y_t^n-a|^2+E\int_t^T||Z_s^n||^2ds+2RE\int_t^T|A_n(Y_s^n)|ds$
\begin{equation}\begin{array}{rcl}&\leq&E |\xi-a|^2+2E\int_t^T(Y_s^n-a)f(s,Y_s^n,Z_s^n)ds\\
&&+2CE\int_t^T|Y_s^n|ds+2C\\
&\leq&E
|\xi-a|^2+CE\int_t^T|Y_s^n-a|^2ds+\frac{1}{2}E\int_t^T||Z_s^n||^2ds+C,
 \end{array}\end{equation}
where we have used the elementary inequality $2ab\leq \beta^2
a^2+\frac{b^2}{\beta^2}$ for all $a,b\geq 0.$ So, we have
\begin{equation}
E|Y_t^n-a|^2+\frac{1}{2}E\int_t^T||Z_s^n||^2ds\leq
C(1+E\int_t^T|Y_s^n-a|^2ds).\end{equation} Gronwall inequality shows
\begin{equation}
E|Y_t^n-a|^2\leq C, \ \forall n.\end{equation} So that
\begin{equation}
E|Y_t^n|^2\leq C, \ \forall n.\end{equation} From (6) and (7), it is
easy to show
\begin{equation}
E\int_0^T (||Z_s^n||^2+|A_n(Y_s^n)|)ds\leq C, \ \forall
n.\end{equation} Bulkholder-Davis-Gundy inequality shows the desired
result of Lemma 3.1.
\end{proof}
\begin{la}
Under the assumptions of Theorem 3.1, there exists a constant
$C_2>0$ such that for all $n\geq 1$
\begin{center}
$E\int_0^T|A_n(Y_s^n)|^2ds\leq C_2.$
\end{center}
\end{la}
\begin{proof}
Without loss of generality, we assume $\Phi\geq 0 ,\Phi(0)=0$. Let
$\psi_n\triangleq \frac{\Phi_n}{n}.$
 For the convexity of $\psi_n$ and It\^{o} formula, we have
\begin{equation}\begin{array}{rcl}\psi_n(Y_t^n)&\leq&\psi_n(\xi)+\int_t^T\nabla\psi_n(Y_s^n)[f(s,Y_s^n,Z_s^n)-A_n(Y_s^n)]ds\\
&&-\sum_{i=1}^\infty\int_t^T \nabla \psi_n(Y_s^n)Z_s^{(i),n}dH_s^i.
 \end{array}\end{equation}
Taking expectation on the both sides, we obtain
\begin{equation}\begin{array}{rcl}E\psi_n(Y_t^n)&\leq&E\psi_n(\xi)+E\int_t^T\nabla\psi_n(Y_s^n)f(s,Y_s^n,Z_s^n)ds\\
&&-E\int_t^T\nabla\psi_n(Y_s^n)A_n(Y_s^n)ds\\
&=&E\psi_n(\xi)+E\int_t^T\nabla\psi_n(Y_s^n)f(s,Y_s^n,Z_s^n)ds\\
&&-\frac{1}{n}E\int_t^T|A_n(Y_s^n)|^2ds.
 \end{array}\end{equation}
Using the elementary inequality $2ab\leq \beta
a^2+\frac{b^2}{\beta}$, we obtain

$E\psi_n(Y_t^n)+\frac{1}{n}E\int_t^T|A_n(Y_s^n)|^2ds$
\begin{equation}\begin{array}{rcl}
&\leq&E\psi_n(\xi)+\frac{1}{2n}E\int_t^T|A_n(Y_s^n)|^2ds+\frac{1}{2n}E\int_t^T|f(s,Y_s^n,Z_s^n)|^2ds.
 \end{array}\end{equation}
Further, we have

$E\psi_n(Y_t^n)+\frac{1}{n}E\int_t^T|A_n(Y_s^n)|^2ds$
\begin{equation}\begin{array}{rcl}
&\leq&E\psi_n(\xi)+\frac{1}{2n}E\int_t^T|A_n(Y_s^n)|^2ds+\frac{1}{2n}E\int_t^T|f(s,Y_s^n,Z_s^n)|^2ds\\
&\leq&E\psi_n(\xi)+\frac{1}{2n}E\int_t^T|A_n(Y_s^n)|^2ds+\frac{1}{2n}E\int_t^T|f(s,Y_s^n,Z_s^n)|^2ds\\
&\leq&E\psi_n(\xi)+\frac{1}{2n}E\int_t^T|A_n(Y_s^n)|^2ds+\frac{C}{n}E\int_t^T|Y_s^n|^2ds\\&&
+\frac{C}{n}E\int_t^T||Z_s^n||^2ds+\frac{C}{n}E\int_t^T|f(s,0,0)|^2ds.
 \end{array}\end{equation}
Lemma 3.1 shows
\begin{center}
$E\psi_n(Y_t^n)+\frac{1}{n}E\int_t^T|A_n(Y_s^n)|^2ds\leq
\frac{C}{n},$
\end{center}
which implies the desired result.
\end{proof}
\begin{la}
Under the assumptions of Theorem 3.1, $(Y^n,Z^n)$ be a Cauchy
sequence in $S^2\times \mathcal{P}^2(l^2).$
\end{la}
\begin{proof}
By It\^{o} formula, we have

$|Y_t^n-Y_t^m|^2+\int_t^T||Z_s^n-Z_s^m||^2ds$
\begin{equation}\begin{array}{rcl}
&=&2\int_t^T(Y_s^n-Y_s^m)[f(s,Y_s^n,Z_s^n)-f(s,Y_s^m,Z_s^m)]ds\\
&&-2\int_t^T(Y_s^n-Y_s^m)(A_n(Y_s^n)-A_m(Y_s^m))ds\\
&&-2\sum_{i=1}^\infty \int_t^T
(Y_s^n-Y_s^m)(Z_s^{(i),n}-Z_s^{((i),m})dH_s^i.
 \end{array}\end{equation}
For

$I=J_n+\frac{1}{n}A_n=J_m+\frac{1}{m}A_m,\ A_m\in
\partial\Phi(J_m(Y_s^m)),\ A_n\in
\partial\Phi(J_n(Y_s^n)),$

we can obtain

$-(Y_s^n-Y_s^m)(A_n(Y_s^n)ds-A_m(Y_s^m))\leq
\frac{1}{4m}|A_n(Y_s^n)|^2+\frac{1}{4n}|A_m(Y_s^m)|^2.$

So, we get

 $E|Y_t^n-Y_t^m|^2+E\int_t^T||Z_s^n-Z_s^m||^2ds$
\begin{equation}\begin{array}{rcl}
&\leq&2CE\int_t^T(|Y_s^n-Y_s^m|^2+|Y_s^n-Y_s^m||Z_s^n-Z_s^m|)ds\\
&&+E\int_t^T(\frac{1}{4m}|A_n(Y_s^n)|^2+\frac{1}{4n}|A_m(Y_s^m)|^2)ds\\
&\leq&
2CE\int_t^T[(1+\beta)|Y_s^n-Y_s^m|^2+\frac{1}{\beta}|Z_s^n-Z_s^m|^2]ds\\
&&+E\int_t^T(\frac{1}{4m}|A_n(Y_s^n)|^2+\frac{1}{4n}|A_m(Y_s^m)|^2)ds.
 \end{array}\end{equation}
Choosing $\beta$ such that $\frac{2C}{\beta}<\frac{1}{2},$ then

$\sup_{0\leq t\leq
T}E|Y_t^n-Y_t^m|^2+\frac{1}{2}E\int_0^T||Z_s^n-Z_s^m||^2ds\leq
C(\frac{1}{n}+\frac{1}{m}).$

Further, Bulkholder-Davis-Gundy inequality shows
\begin{equation}E\sup_{0\leq t\leq
T}|Y_t^n-Y_t^m|^2+\frac{1}{2}E\int_0^T||Z_s^n-Z_s^m||^2ds\leq
C(\frac{1}{n}+\frac{1}{m}),\end{equation} which shows our desired
result.
\end{proof}
In order to obtain the existence of solutions to Eq. (2), we give
the following Lemma appeared in Saisho \cite{Sa}.

\begin{la}Let $\{K^{(n)}, n\in N\}$ be a family of continuous functions of finite
variation on $R^+$. Assume that:

(i) $\sup_n |K^{(n)}|_t \leq C_t <\infty, \ 0\leq t <\infty;$

 (ii) $\lim_{n\rightarrow \infty} K^{(n)} = K \in C([0,+\infty);R).$

  Then K is of finite variation. Moreover, if $\{ f^{(n)}, n \in N\}$ is a family of continuous functions
such that $\lim_{n\rightarrow \infty} f^{(n)} = f\in
C([0,+\infty);R),$  then

$\lim_{n\rightarrow
\infty}\int_s^t<f_u^{(n)},dK_u^{n)}>=\int_s^t<f_u,dK_u>,\ \forall
0\leq s\leq t<\infty.$
\end{la}
{\bf Proof of Theorem 3.1} \textsl{Existence.} Lemma 3.3 shows that
$(Y^n,Z^n)$ is a Cauchy sequence in space  $S^2\times
\mathcal{P}^2(l^2).$ We denote its limit as $(Y,Z)$. At the same
time, it is easy to verify that $(K^n)_{n\in N}$ converges uniformly
in $L^2(\Omega)$ to the process $K_\cdot=\lim_{n\rightarrow
\infty}\int_0^\cdot A_n(Y_s^n)ds$, that is $E\sup_{0\leq t \leq
T}|K_t^n-K_t|^2=0.$

Denote $H^1(0,T;R)$ be the Sobolev space consisting of all
absolutely continuous functions with derivative in $L^2(0,T)$. Lemma
3.2 shows
\begin{center}
$\sup_{n\in N}E||K^n||^2_{H^2(0,T;R)}<\infty,$
\end{center}
which implies that the sequence $K^n$ is bounded in
$L^2(\Omega;H^1(0,T;R))$. So, there exists an absolutely continuous
function $K\in L^2(\Omega;H^1(0,T;R))$ which is the weak limit of
$K^n$. Further, $\frac{dK_t}{dt}=V_t,$ where $-V_t\in \partial
\Phi(Y_t).$

In the following, we verify that $(Y,Z,K)$ is the unique solution to
Eq. (2). Taking a subsequence, if necessary, we can show

$\sup_{0\leq t \leq T}|Y_t^n-Y_t|\rightarrow 0,$

$\sup_{0\leq t \leq T}|K_t^n-K_t|\rightarrow 0.$

It follows that $Y_t$ is c\`{a}gl\`{a}g and $K_t$ is continuous.
Further, let $(\alpha,\beta)$ be a c\`{a}gl\`{a}g process with
values in $Gr(\partial \Phi),$  it holds
\begin{center}
$<J_n(Y_t^n)-\alpha_t,dK_t^n+\beta_tdt>\leq 0.$
\end{center}
Since $J_n$ is a contraction and $Y^n$ uniformly converges to $Y$ on
$[0,T]$. It follows that $J_n(Y_t^n)$ converges to $pr(Y)$
uniformly, where $pr$ denotes the projection on
$\overline{Dom(\Phi)}$. Lemma 3.4 shows
\begin{center}
$<pr(Y_t)-\alpha_t,dK_t+\beta_tdt>\leq 0.$
\end{center}
In order to complete the proof of the existence, we need to verify
the following
\begin{center}
$P(Y_t\in \overline{Dom(\Phi)},0\leq t \leq T)=1.$
\end{center}
For the right continuity of $Y$, it suffices to prove
\begin{center}
$P(Y_t\in \overline{Dom(\Phi)})=1,\ 0\leq t \leq T.$
\end{center}

If not so, there exists $0<t<T$ and $B_0\in \mathcal {F}$ such that
$P(B_0)>0$ and $Y_{t_0}\not\in \overline{Dom(\Phi)}, \forall \omega
\in B_0.$ By the right continuity, there exists $\delta>0,\ B_1\in
\mathcal {F}$ such that $P(B_1)>0,\ Y_t \not\in
\overline{Dom(\Phi)}$ for $(\omega,t)\in B_1\times
[t_0,t_0+\delta].$

Using the fact $\sum_{n\in N}E\int_0^T|A_n(Y_s^n)|ds<\infty$ and
Fatou Lemma, we get
\begin{center}
$\int_{B_1}\int_{t_0}^{t_0+\delta}\liminf_{n\rightarrow
\infty}|A_n(Y_s^n)|dsdP<\infty,$
\end{center}
which is impossible for $\liminf_{n\rightarrow
\infty}|A_n(Y_s^n)|=\infty$ on the set $B_1\times [t_0,t_0+\delta].$

\textsl{Uniqueness.} Let $(Y_t,Z_t,K_t)_{0\leq t \leq T}$ and
$(Y_t^\prime,Z_t^\prime,K_t^\prime)_{0\leq t \leq T}$ be two
solutions for Eq. (2). Define
\begin{center}
$(\triangle Y_t,\triangle Z_t, \triangle K_t)_{0\leq t \leq
T}=(Y_t-Y_t^\prime,Z_t-Z_t^\prime,K_t-K_t^\prime)_{0\leq t \leq T}.$
\end{center}
It\^{o} formula shows

 $E|\triangle Y_t|^2+E\int_t^T||\triangle Z_t||^2ds$
\begin{equation}\begin{array}{rcl}
&=&2E\int_t^T\triangle Y_s
[f(s,Y_s,Z_s)-f(s,Y_s^\prime,Z_s^\prime)]ds+2E\int_t^T\triangle
Y_sd\triangle K_s.
 \end{array}\end{equation}
Since $\partial \Phi$ is monotone and $-\frac{dK_t}{dt}\in \partial
\Phi(Y_t),\ -\frac{dK_t^\prime}{dt}\in \partial \Phi(Y_t^\prime),$
we obtain
\begin{center} $E\int_t^T\triangle Y_sd\triangle K_s\leq
0.$\end{center} Further,

 $E|\triangle Y_t|^2+E\int_t^T||\triangle Z_t||^2ds\leq CE\int_t^T|\triangle Y_s|^2ds
 +\frac{1}{2}E\int_t^T||\triangle Z_s||^2ds.$

Gronwall inequality shows the uniqueness of the solutions to Eq.
(2).

\section{Applications}
In this section, we study the link between RBSDEs driven L\'{e}vy
processes and the solution of a class of PDIIs. Suppose that our
L\'{e}vy process has the form of $L_t=at+X_t $ where $X_t$ is a pure
jump L\'{e}vy process with L\'{e}vy measure $\nu(dx).$

In order to attain our main result, we give a Lemma appeared in
\cite{Nu2}.
\begin{la}
Let $c:\Omega\times [0,T]\times \mathbb{R}\rightarrow \mathbb{R}$ be
a measurable function such that
\begin{center}
$|c(s,y)|\leq a_s(y^2\wedge |y|)\ a.s.,$
\end{center}
where $\{a_s,s\in [0,T]\}$ is a non-negative predictable process
such that $E\int_0^Ta_s^2ds<\infty.$ Then, for each $0\leq t \leq
T,$ we have

$\begin{array}{rcl} \sum_{t<s\leq T}c(s,\Delta
L_s)&=&\sum_{i=1}^\infty
\int_t^T<c(s,\cdot),p_i>_{L^2(\nu)}dH_s^{(i)}\\
&&+\int_t^T\int_{\mathbb{R}}c(s,y)\nu(dy)ds.\end{array}$
\end{la}
Consider the following coupled RBSDEs:
\begin{equation}
\begin{array}{rcl}
Y_t&=&h(L_T)+\int_t^Tf(s,L_s,Y_s,Z_s)ds+K_T-K_t-\sum_{i=1}^\infty
\int_t^TZ_s^{(i)}dH_s^{(i)},\ 0\leq t \leq T,
\end{array}
\end{equation}
where $E|h(L_T)|^2<\infty.$

Define
\begin{center}
$u^1(t,x,y)=u(t,x+y)-u(t,x)-\frac{\partial u}{\partial x}(t,x)y,$
\end{center}
where $u$ is the solution of the following PDIIs:
\begin{equation}
\left\{\begin{array}{ll}
\begin{array}{rcl} \frac{\partial u}{\partial
t}(t,x)+a^\prime \frac{\partial u}{\partial
x}(t,x)+f(t,u(t,x),\{u^{(i)}(t,x)\}_{i=1}^\infty)
&&\\+\int_{\mathbb{R}}u^1(t,x,y)\nu(dy)\in \partial
\Phi(u(t,x)),\end{array}
\\
u(T,x)=h(x)\in \overline{Dom(\Phi)},  \end{array}\right.
\end{equation}
 where $a^\prime =a+\int_{\{|y|\geq
1\}}y\nu(dy)$ and
\begin{center}
$u^{(1)}(t,x)=\int_{\mathbb{R}}u^1(t,x,y)p_1(y)\nu(dy)+\frac{\partial
u}{\partial x}(t,x)(\int_{\mathbb{R}}y^2\nu(dy))^{\frac{1}{2}},$
\end{center}
and for $i\geq 2$
\begin{center}
$u^{(i)}(t,x)=\int_{\mathbb{R}}u^1(t,x,y)p_i(y)\nu(dy).$
\end{center}

Suppose that $u$ is $\mathcal {C}^{1,2}$ function such that
$\frac{\partial u}{\partial t}$ and $\frac{\partial^2 u}{\partial
x^2}$ is bounded by polynomial function of $x,$ uniformly in $t$.
Then we have the following
\begin{thm}
The unique adapted solution of (19) is given by
\par
$Y_t=u(t,L_t),$
\par
$K_t=\int_0^tV_sds,\ -V_s\in \partial \Phi(u(t,L_t)),$
\par
$Z_t^{(i)}=\int_{\mathbb{R}}u^1(t,L_{t-},y)p_i(y)\nu(dy),\ i\geq 2,$
\par
$Z_t^{(1)}=\int_{\mathbb{R}}u^1(t,L_{t-},y)p_1(y)\nu(dy)+\frac{\partial
u}{\partial x}(t,L_{t-})(\int_{\mathbb{R}}
y^2\nu(dy))^{\frac{1}{2}}.$
\end{thm}
\begin{proof}
Applying It\^{o} formula to $u(s,L_s)$, we obtain

\hspace{1cm}$u(T,L_T)-u(t,L_t)$
\begin{equation}
\begin{array}{rcl}  &=& \int_t^T \frac{\partial
u}{\partial s}(s,L_{s-})ds+\int_t^T \frac{\partial u}{\partial
x}(s,L_{s-})dL_s\\
&&+\sum_{t<s\leq T}[u(s,L_s)-u(s,L_{s-})-\frac{\partial u}{\partial
x}(s,L_{s-})\triangle L_s] .\end{array}
\end{equation}
Lemma 4.1 applied to $u(s,L_{s-}+y)-u(s,L_{s-})-\frac{\partial
u}{\partial x}(s,L_{s-})y$ shows

$\sum_{t<s\leq T}[u(s,L_s)-u(s,L_{s-})-\frac{\partial u}{\partial
x}(s,L_{s-})\triangle L_s]$
\begin{equation}
\begin{array}{rcl}  &=& \sum_{i=1}^\infty
\int_t^T(\int_{\mathbb{R}}u^1(s,L_{s-},y)p_i(y)\nu(dy))dH_s^{(i)}+\int_t^T
\int_{\mathbb{R}}u^1(s,L_{s-},y)\nu(dy)ds.\end{array}
\end{equation}
Note that
\begin{equation}
L_t=Y_t^{(1)}+tEL_1=(\int_{\mathbb{R}}y^2\nu(dy))^{\frac{1}{2}}H_t^{(1)}+tEL_1,
\end{equation}
where $EL_1=a+\int_{\{|y|\geq 1\}}y\nu(dy).$

Hence, substituting (22) and (23) into (21) yields

$ h(L_T)-u(t,L_t)$

 $\begin{array}{rcl}&=& \int_t^T[\frac{\partial
u}{\partial s}(s,L_{s-})+a\frac{\partial u}{\partial
x}(s,L_{s-})+\int_{\{|y|\geq
1\}}y\nu(dy)+\int_{\mathbb{R}}u^1(s,L_{s-},y)\nu(dy)]ds\\
&&+\int_t^T[u^1(s,L_{s-},y)p_1(y)\nu(dy)+\frac{\partial
u}{\partial x}(s,L_{s-})(\int_{\mathbb{R}}y^2\nu(dy))^{\frac{1}{2}}]dH_s^{(1)}\\
&&+\sum_{i=2}^\infty
\int_t^T(\int_{\mathbb{R}}u^1(s,L_{s-},y)p_i(y)\nu(dy))dH_s^{(i)}.
\end{array}$

From which we get the desired result of the Theorem.
\end{proof}

{\bf\large Acknowledgements}

 The authors would like to
 thank the Australian Research Council for funding this project through Discovery Project DP0770388 and
 The Committee of National
Natural Science Foundation in China for the Project 10726075.

\begin{center}
~~\\
\textbf{References}
\end{center}

\end{document}